\documentclass[12pt,a4paper]{article}
\usepackage{fullpage}
\usepackage{authblk}

\usepackage{mathrsfs}
\usepackage{amsfonts}
\usepackage{amssymb}
\usepackage{array}
\usepackage{eurosym}
\usepackage{bm}
\usepackage{multirow,bigstrut}
\usepackage{enumerate}
\usepackage{tikz}
\usepackage{graphicx}
\usepackage{stfloats}
\usepackage{array}
\usepackage{amsmath}
\usepackage{titlesec}
\usepackage[colorlinks,
            linkcolor=red,
            anchorcolor=blue,
            citecolor=blue
            ]{hyperref}

\titleformat*{\section}{\large\bfseries}
\titleformat*{\subsection}{\normalsize\bfseries}

\newenvironment{proof}{\noindent{\em \textbf{Proof.}}}{\quad \hfill$\Box$\vspace{2ex}}

\newtheorem{theorem}{Theorem}[section]

\newtheorem{remark}[theorem]{Remark}
\newtheorem{corollary}[theorem]{Corollary}
\newtheorem{lemma}[theorem]{Lemma}
\newtheorem{example}[theorem]{Example}
\newtheorem{proposition}[theorem]{\bf Proposition}
\numberwithin{equation}{section}

\def \H {\mathbb{H}}

\def \C {\mathbb{C}}

\def\qi { \bm {i}}
\def\qj { \bm{j}}
\def\qk { \bm{k}}

\def \bvphi {\bm{\varphi}}

\def \bmeta {\bm{\eta}}
\def \bxi {\bm{\xi}}

\def \bx {\bm{x}}

\def \bu {\bm{u}}
\def \bv {\bm{v}}
\def \bw {\bm{w}}
\def \pL {\mathcal{L}}
\newcommand{\norm}[1]{\left\lVert#1\right\rVert}
\newcommand{\abs}[1]{\left|#1\right|}

\title{\Large \textbf{Sampling expansions associated with quaternion difference equations}}

\author[a,b]{\small Dong Cheng\thanks{chengdong720@163.com}}

\author[c]{\small  Kit Ian Kou\thanks{kikou@umac.mo}}

\author[d]{\small Yonghui Xia \thanks{xiadoc@163.com}}
\author[e]{\small Junfeng Xu  \thanks{xujunf@gmail.com}}

\affil[a]{\small{Research Center for Mathematics and Mathematics Education, Beijing Normal University at Zhuhai, 519087, China}}
\affil[b]{\small{Laboratory of Mathematics and Complex Systems (Ministry of Education), School of Mathematical Sciences, Beijing Normal University, Beijing, 100875, China}}
\affil[c]{\small{Department of Mathematics, Faculty of Science and Technology, University of Macau, Macao, China}}
\affil[d]{\small{Department of Mathematics, Zhejiang Normal University, Jinhua, 321004, China}}
\affil[e]{\small{Department of Mathematics, Wuyi University, Jiangmen, 529020, China}}
\date{}

\begin{document}
  \maketitle
\begin{abstract}
\normalsize
Starting with a quaternion difference equation with boundary conditions,   a parameterized   sequence  which is complete in finite dimensional quaternion Hilbert space is derived. By  employing the parameterized   sequence as the kernel of  discrete transform, we form a quaternion function space whose elements    have sampling expansions. Moreover, through   formulating  boundary-value problems, we make a connection between a class of tridiagonal quaternion matrices and polynomials with quaternion coefficients. We show that for a  tridiagonal symmetric quaternion matrix, one can always associate a quaternion characteristic  polynomial whose roots are eigenvalues of the matrix. Several examples are   given to illustrate the results.

\end{abstract}


\begin{msc}
 39A12,  11R52,	41A05, 12E05.
\end{msc}

\section{Introduction}\label{S1}
Sampling theorems for general integral transforms other than the Fourier one were discussed by Kramer \cite{kramer1959generalized}. Kramer's   result generalizes the Shannon sampling theorem and it can be flexibly applied to many physical and engineering applications.  A necessary condition for Kramer's sampling theorem is that the kernel of integral transform has to be capable of   generating an   orthogonal  basis in a certain Hilbert space. It is known that the kernels in Kramer's   theorem can be extracted from some boundary-value problems of differential equations (see e.g.,\cite{zayed1993advances}). Accordingly, sampling theorems  associated with several types of  boundary-value problems were extensively investigated \cite{zayed1990lagrange,zayed1991kramer}.  Based on difference operators, the author in \cite{Annaby199855} provided a discrete version of Kramer's theorem, and numerous studies of sampling expansions associated with difference equations were introduced in a series of papers   \cite{annaby1998finite,annaby1999sampling,abd2018difference}.

The purpose of this paper is to derive sampling expansions associated with quaternion difference equations. On the one hand, the theory of quaternion dynamic equations (including quaternion difference equations) \cite{wilczynski2009quaternionic,zhang2011global,cheng2018unified} has received a lot of attention recently because many physical problems can be described as quaternion dynamic models \cite{kouxialinear2018,Cheng2020Floquetq,liwangcauchy2020}.  On the other hand, quaternion algebra has shown advantages over real and complex  within color image processing, computer graphics, and multidimensional signal processing, etc., for  modelling of
rotation and  cross-information between multichannel  and high dimensional data  (see
for instance, \cite{xu2015vector,vince2011quaternions,wang2016polarimetric}). In particular,  the hypercomplex sampling scheme plays an important role in  engineering applications \cite{SCHUYLER201320}. 
Although there are sampling theorems concerning quaternion-valued functions defined on $\mathbb{R}^2$ \cite{cheng2017novel,chengkou2018generalized}, only the uniform sampling case is considered.  In this paper,
we investigate  the relationship between  quaternion dynamic equations and sampling theory. Some novel results about sampling expansions   for quaternion-valued functions are obtained by means of the latest approaches used in the theory of quaternion dynamic equations. In contrast to the previous results of quaternionic sampling theory, the quaternion-valued functions considered in the present paper  are defined on quaternion skew field   and the sample points involved in interpolation formulas are non-uniformly spaced in $\mathbb{R}^4$.

 In the quaternion analysis setting, the study of difference equations \cite{cheng2018unified},   matrix theory \cite{zhang1997quaternions,farenick2003spectral} and  zeros computing of polynomials \cite{pogorui2004structure,janovska2010note} is  essentially different from the traditional case due to the non-commutative property of quaternions. To study the sampling theory for the quaternion-valued functions  defined on quaternion skew field (or equivalently  $\mathbb{R}^4$), we cannot directly use the related results  such as spectral theorem of matrices and structure of zeros of polynomials  in the real or complex case. Thus  formulating sampling  expansions for quaternion-valued functions has to be based on the distinctive structure of quaternions. This paper involves eigenvalue problem of quaternion matrices, spectral property of operators in quaternion Hilbert spaces and zeros of quaternion polynomials.  Starting with a quaternion difference equation with boundary conditions,   a parameterized   sequence  which is complete in finite dimensional quaternion Hilbert space is obtained. By  employing the parameterized   sequence as the kernel of  discrete transform, we form a quaternion function space whose elements    have sampling expansions.
 The sample points are not only the eigenvalues of a certain quaternion matrix but also the zeros of a relevant quaternion polynomial. Moreover, we show that for a  tridiagonal symmetric quaternion matrix, one can always associate a quaternion    polynomial whose roots are eigenvalues of the matrix.
The main contributions of this paper are highlighted as follows:
\begin{enumerate}
	\item We present the sampling expansions associated with quaternion difference equations (Theorem \ref{samplthm}).  The quaternion-valued functions considered in the   paper  are defined on quaternion skew field   and the sample points involved in interpolation formulas are non-uniformly spaced in $\mathbb{R}^4$.
	\item It is shown that there are infinitely many sampling expansions for the functions satisfying certain conditions (Proposition \ref{nonunique}). We show how to construct different sampling expansions for such a function (see the proof of Proposition \ref{nonunique} and Example \ref{exsampling1}). Furthermore, one can derive a sampling expansion  for a function  from an existing expansion of another function  by making use of the relationship between them (Example \ref{exsampling1}). 
	\item  We investigate the characteristic polynomials for tridiagonal symmetric  quaternion matrices (Theorem \ref{characterpoly}). It is shown that     $\gamma$ is an eigenvalue of such a matrix  if and only if $\gamma$ is similar with a zero of the characteristic polynomial.
\end{enumerate}

The rest of the paper is organized as follow.  In Section \ref{S2}, some useful results of quaternion algebra  are reviewed. Besides, a lemma  of quaternion polynomial is presented.   Section \ref{S3} is devoted to the sampling expansions associated with quaternion difference equations.  In Section \ref{S4}, we investigate the characteristic polynomials for  tridiagonal symmetric quaternion matrices. Finally,  conclusions are drawn at the end of the paper.

\section{Preliminaries}\label{S2}

\subsection{Quaternions and matrices of quaternions}\label{S21}
The skew field of quaternions \cite{sudbery1979quaternionic} denoted by $\H$ is the four-dimensional algebra over $\mathbb{R}$ with basis $\{1,\qi,\qj,\qk\}$. The elements $\qi$, $\qj$ and $\qk$ obey the Hamilton's multiplication rules
\begin{equation*}
\qi\qj=-\qj\qi=\qk,\ \qj\qk=-\qk\qj=\qi,\ \qk\qi=-\qi\qk,\ \qi^2=\qj^2=\qi\qj\qk=-1.
\end{equation*}
For each quaternion $q=q_0+q_1\qi+q_2\qj+q_3\qk$, its conjugate is defined by $\overline{q}=q_0-q_1\qi-q_2\qj-q_3\qk$, then $q\overline{q}=\overline{q}q=\sum_{m=0}^3q_m^2$ and its norm is given by $\abs{q}=\sqrt{q\overline{q}}$.
Using the conjugate and norm of $q$, one can define the inverse of 
$q\in \H \setminus \{0\}$ by $q^{-1}=\overline{q}/\abs{q}^2$. Observe that the set $\mathbb{C}$ of complex numbers appears as a sub-algebra of $\H$: $\mathbb{C}=\mathrm{Span}_{\mathbb{R}}\{1,\qi\}$, thus we will view $\mathbb{C}$ as a subset of $\H$.

Let $M_{m\times n}(\H)$, simply $M_n(\H)$ when $m=n$, denote the set of all $m$ by $n$ matrices with entries from $\H$. Just as with the complex case, a square matrix $A\in M_n(\H)$ is said to be normal if $A^*A=AA^*$, unitary if $AA^*=I$ and invertible if $AB=BA=I$ for some $B\in M_n(\H)$. Here $A^*$ is the conjugate transpose of $A$. 
The concept of eigenvalues for quaternion matrices is somewhat different from the complex case.  Owing to the non-commutativity of quaternions, there are two types (left and right) of eigenvalues for quaternion matrices.   A vector $\bxi\in \H^n\setminus \{0\}$ is said to be a right (left) eigenvector of $A$ corresponding to the right (left) eigenvalue $\lambda\in\H$ provided that 
\begin{equation}\label{defeigvalue}
A\bxi=\bxi\lambda \quad (A\bxi= \lambda \bxi)
\end{equation}
holds. In the sequel we will only consider right eigenvalues and right eigenvectors,  so we will   use terminology eigenvalues and eigenvectors for simplicity.

 A matrix $B_1$ is said to be similar to a matrix $B_2$ if $B_2=S^{-1}B_1S$ for some non-singular matrix $S$. 
In particular, we say that two quaternions $p,q$ are similar if $p=\alpha^{-1}q\alpha$ for some nonzero   $\alpha\in\H$.
From \cite{zhang1997quaternions,rodman2014topics}, we know that any matrix $A\in M_n(\H) $ has exactly $n$ eigenvalues (including multiplicity) which are complex numbers with nonnegative imaginary parts. These eigenvalues are called standard eigenvalues. Like the complex case, any two similar quaternion matrices have the same eigenvalues. We denote the totality of eigenvalues of $A$ by $\sigma(A)$, the spectrum of $A$. It is easy to see that
if (\ref{defeigvalue})  holds, then $A\bxi\alpha=(\bxi\alpha)(\alpha^{-1}\lambda\alpha)$ for
all nonzero  $\alpha\in\H$. This means that $\bxi\alpha$ is a eigenvector of $A$ corresponding to eigenvalue $\alpha^{-1}\lambda\alpha$ rather than $\lambda$.
 For any  $q\in \H$,  its similarity orbit \cite{farenick2003spectral} is defined by
\begin{equation*}
\theta(q):=\left\{\alpha^{-1}q\alpha: \alpha\in \H\setminus\{0\}\right\}=\left\{\alpha q\overline{\alpha}: \alpha\in\H, \abs{\alpha}=1\right\}.
\end{equation*}
It follows that if $\theta(q)\cap \sigma(A)\neq \emptyset$, then $\theta(q)\subset \sigma(A)$. The similarity orbit $\theta(q)$ contains infinitely many elements for $q\in\H\setminus \mathbb{R}$, but only two of them are complex. Furthermore,  two similarity orbits are disjoint, as described by the following lemmas.
\begin{lemma}\cite{farenick2003spectral}
For every $q \in\H\setminus \mathbb{R}$, there is a non-real $z\in\C$ such that $\theta(q)\cap \C =\{z,\overline{z}\}$.
\end{lemma}
\begin{lemma}\cite{cheng2017novel}
	If $\theta(p)\cap\theta(q)\neq \emptyset$, then $\theta(p)=\theta(q)$.
\end{lemma}

\subsection{Quaternion Hilbert spaces}\label{S22}
An abelian group $H$ is a right $\H$-module \cite{farenick2003spectral} if there is a right scalar multiplication map
$(\bu,\alpha)\mapsto \bu\alpha$ from $H\times \H$ into $H$ such that for all $\bu,\bv\in H$ and $\alpha,\beta\in\H$
\begin{equation*}
\bu(\alpha+\beta)=\bu\alpha+\bu\beta,\ (\bu+\bv)\alpha= \bu\alpha+\bv\alpha,\   \bu(\alpha\beta)=(\bu\alpha)\beta,\ \bu 1=\bu.
\end{equation*}
A subset $V=\{\bu_1,\bu_2,\dots,\bu_m\} $ of $H$ is said to be (right) $\H$-independent if 
\begin{equation*}
\sum_{l=1}^m \bu_l\alpha_l=0,\ \alpha_l\in\H \quad \text{implies that}\quad  \alpha_1=\alpha_2=\cdots=\alpha_m=0.
\end{equation*}
If any element of $H$ can be expressed by a right $\H$-linear combination of $V\subset H$, then $V$ is called a basis of $H$ and the dimension of $H$ is $m$.

A right $\H$-module $H$
is called a quaternion pre-Hilbert space if there exists a quaternion-valued
function (inner product) $\langle\cdot, \cdot\rangle : H \times H \to \H$ such that for all $\bu,\bv,\bw\in H$ and $\alpha,\beta\in\H$:
\begin{enumerate}
	\item $\langle \bu,\bv \rangle=\overline{\langle \bv,\bu \rangle}$,
	\item $\langle \bu,\bv\alpha+\bw\beta\rangle =\langle \bu,\bv \rangle \alpha+\langle \bu,\bw \rangle \beta$,
	\item $\langle \bu,\bu\rangle\geq0$ and $\langle \bu,\bu\rangle=0$ if and only if $\bu=0$.
\end{enumerate}
The function $u\mapsto \norm{\bu}=\sqrt{\langle \bu,\bu\rangle}$ is a norm on $H$. Under this norm, the Cauchy-Schwartz inequality and triangular  inequality (see \cite{brackx1982clifford}) hold as 
$\abs{\langle \bu,\bv\rangle}\leq \langle \bu,\bu\rangle \langle \bv,\bv\rangle$ and $ \norm{\bu+\bv}\leq \norm{\bu}+\norm{\bv}$ respectively. Moreover, if $H$ is complete under the norm $\norm{\cdot}$, then it is called a quaternion Hilbert space. For any $n$-dimensional Hilbert space $H$,  by the quaternion version of Gram-Schmidt theorem \cite{farenick2003spectral}, there exists a basis $V=\{\bu_1,\bu_2,\dots,\bu_n\} $ of $H$ such that $\norm{\bu_k}=1$ for every $1\leq k \leq n$ and $\langle \bu_s, \bu_t\rangle=0$ for $s\neq t$. Such a basis is called an orthonormal basis. For every $\bv\in H$, it can be expanded as
$\bv=\bu_1\langle \bu_1,\bv\rangle +\bu_2\langle \bu_2,\bv\rangle+\cdots+\bu_n\langle \bu_n,\bv\rangle$.

A right $\H$-linear operator is a function $\mathcal{T}: H\to H$ such that
$\mathcal{T}(\bu\alpha+\bv\beta)=\mathcal{T}(\bu)\alpha+\mathcal{T}(\bv)\beta$ for all  $\bu,\bv\in H$ and $\alpha,\beta\in\H$.  Such an operator is also called a endomorphism, so we denote by $E_n(H)$  the set of all endomorphisms on $H$. For every $\mathcal{T}\in E_n(H)$, the Riesz representation theorem (see e.g., \cite{brackx1982clifford,farenick2003spectral}) guarantees that there exists a unique operator $\mathcal{T}^*\in E_n(H)$, which is called the adjoint of $\mathcal{T}$, such that for all $\bu,\bv\in H$, $\langle\mathcal{T}\bu,\bv\rangle=\langle \bu,\mathcal{T}^*\bv\rangle$. As for quaternion matrices, an operator $\mathcal{T}\in E_n(H)$ is said to be  self-adjoint if $\mathcal{T}=\mathcal{T}^*$, normal if $\mathcal{T}\mathcal{T}^*=\mathcal{T}^*\mathcal{T}$ and unitary if $\mathcal{T}\mathcal{T}^*=\mathcal{T}^*\mathcal{T}=\mathcal{I}$.

We recall the spectral theorem for normal operators \cite{farenick2003spectral}
as follows.
\begin{theorem}\label{spectral1}
Suppose that $H$ is an $n$-dimensional quaternion Hilbert space. Then $\mathcal{T}\in E_n(H)$ is normal if and only if there is an orthonormal basis $V=\{\bu_1,\bu_2,\dots,\bu_n\}\subset H$ and $\lambda_1,\lambda_2,\dots,\lambda_n\in\C^+$ such that $\mathcal{T}\bu_k = \bu_k\lambda_k$ for every $1\leq k \leq n$.
\end{theorem}

Under the usual vector-scalar multiplication $\langle\bxi,\alpha\rangle\mapsto \bxi\alpha$, $\H^n$ is a right $\H$-module. Furthermore, $\H^n$ is a quaternion Hilbert space if it is embedded by inner product $\langle\bxi,\bmeta \rangle=\bxi^*\bmeta$.
Let $A\in M_n(\H), \bxi\in\H^n$ and define $\mathcal{T}_A = A\bxi$. Then we have the spectral theorem for normal matrices as follows.

\begin{theorem}\label{spectral2}
 The matrix $A\in M_n(H)$ is normal if and only if there is a unitary matrix $U\in M_n(\H)$ such that $U^*AU=D$, where   $D$ is a diagonal matrix and its entries are complex numbers with nonnegative imaginary parts. 
\end{theorem}

\subsection{Quaternion polynomials}\label{S23}

The polynomials with quaternion coefficients located on only  left side of powers are called simple quaternion polynomials \cite{janovska2010note}. Let
\begin{equation}\label{simplepolynomial}
p_n(z):=\sum_{k=0}^{n}c_kz^{k}, \quad c_k\in \H, \ 0\leq k\leq n,\ c_0c_n\neq 0
\end{equation}
be a given simple quaternion polynomial of degree $n$. It was shown in \cite{pogorui2004structure} by Pogorui and Shapiro that the polynomials of type (\ref{simplepolynomial}) may have two types of zeros: isolated and spherical zeros. 
Let $z_0$ be a zero of $p_n$ defined by (\ref{simplepolynomial}). If $z_0\notin \mathbb{R}$ and  has the property that $p_n(z)=0$ for all $z\in \theta(z_0)$, then it is called a spherical zero. Otherwise, $z_0$ is called an isolated zero. If the zero set of $p_n(z)$ intersects only with $n_1$ similarity orbits, we say that the number of zeros of $p_n(z)$ is $n_1$.  The authors in \cite{janovska2010note} proved that  $n_1\leq n$,   and they also presented an effective algorithm for finding all zeros including their types without using iterations.

Let $p_n(z,s)=\sum_{k=0}^{n}c_ksz^{k}$, we introduce a lemma concerning the zeros of $p_n(z,s_1)$ and $p_n(z,s_2)$ for $s_1\neq s_2$. 
\begin{lemma}\label{zerospzs}
Let $p_n(z,s)$ be given above. Then $z_0$ is a zero of $p_n(z,s_1)$ if and only if $s^{-1}z_0s$ is a zero of $p_n(z,s_1)$,  where $s=s_1^{-1}s_2$.
Besides $p_n(z,s_1)$ and $p_n(z,s_2)$ have the same real and spherical zeros.
Let  $Z_{iso}(s)$ be the totality of non-real isolated zeros of  $p_n(z,s)$. Then $$Z_{iso}(s_2)=\left\{s^{-1}zs: s=s_1^{-1}s_2,\ z\in Z_{iso}(s_1)\right\}.$$ In particular, 
$Z_{iso}(s_2)= Z_{iso}(s_1)$ provided that $s_1^{-1}s_2\in\mathbb{R}$.
\end{lemma}
\begin{proof}
	For any fixed $z_0\in\H$, we have that $(s^{-1}z_0s)^k=s^{-1}z_0^ks$. It follows that 
	\begin{equation*} 
p_n(s^{-1}z_0s,s_2)=\sum_{k=0}^{n}c_ks_2(s^{-1}z_0s)^k = \sum_{k=0}^{n}c_ks_2s^{-1}z_0^ks =  \sum_{k=0}^{n}c_ks_2s_2^{-1}s_1z_0^ks=p_n(z_0,s_1)s.
	\end{equation*}
	Therefore, $z_0$ is a zero of	$p_n(z,s_1)$ if and only if $s^{-1}z_0s$ is a zero of $p_n(z,s_1)$. If $z_0$ is real, then  $s^{-1}z_0s=z_0$. Thus $p_n(z,s_1)$ and $p_n(z,s_2)$ have the same real zeros. Since $z_0$ and $s^{-1}z_0s$ belong to the similarity orbit $\theta(z_0)$, then $p_n(z,s_1)$ and $p_n(z,s_2)$ have the same spherical zeros. They have the same non-real isolated zeros if  $s_1^{-1}s_2\in\mathbb{R}$.
\end{proof}

\section{The sampling theorem}\label{S3}
We  study the quaternion difference equation
\begin{equation}\label{diffequ}
b(k)x(k+1)+a(k)x(k)+b(k-1)x(k-1) =x(k)\lambda,\quad k=1,2,\dots,N,
\end{equation}
where   $a(k)\in\H,\  b(k)\in\H\setminus\{0\}$ for $0\leq k\leq N$ and  $\lambda\in \H$ is a parameter.  The equations of type (\ref{diffequ}), which have two-sided coefficients, are difficult to  solve \cite{Caikou2018}. To solve (\ref{diffequ}), we need to consider  boundary conditions of the form
\begin{align}
&x(0)+h_1x(1)=0 \label{BV1}\\
&x(N+1)+h_2x(N)=0 \label{BV2}
\end{align}
where $h_1, h_2$ are quaternion numbers. Note that $x(1)$ and $x(N)$ cannot be zero, otherwise there is only a trivial solution ($x(k)=0$) for the boundary value problem (BVP) (\ref{diffequ})--(\ref{BV2}). Therefore if $h_1=0$ (which implies that $x(0)=0$), we have to restrict $x(1)\neq 0$.  Similar considerations are needed for the cases when $h_2=0$ or $h_1=h_2=0$.

 For fixed $h_1, h_2\in\H$, we define an operator $\pL$ for any $\bx=[x(1),x(2),\dots,x(N)]^{\top}\in \H^n$ as
\begin{equation}\label{diffoperator}
(\pL\bx)(k) := b(k)x(k+1)+a(k)x(k)+b(k-1)x(k-1), \quad k=1,2,\dots,N,
\end{equation}
where $x(0), x(N+1)$ involved in $\pL$ are taken from  (\ref{BV1}) and (\ref{BV2}).
In the traditional case, the BVP (\ref{diffequ})--(\ref{BV2}) is called a regular Sturm-Liouville problem if  $a(k), b(k), h_1, h_2$ are restricted to be real and the corresponding operator $\pL$ is self-adjoint. In the present study, $\pL$ is not necessarily to be self-adjoint.  In fact, by Theorem \ref{spectral1}, we know that if $\pL$ is normal, then it can produce an orthonormal basis whose elements are solutions of 
(\ref{diffequ}).

Since $\pL$ is a right $\H$-linear operator on $\H^N$, it has a matrix form  
\begin{equation*}
L=\begin{pmatrix}
d_1&b(1)  &  &  & &\\ 
b(1)& a(2) & b(2) & & \text{\huge 0}& \\ 
& b(2) &a(3) &(b(3))&&\\
&&\ddots&\ddots&\ddots&\\
& \text{\huge 0} &  & b(N-2) & a(N-1)&b(N-1)\\ 
&  &  & &b(N-1) & d_2
\end{pmatrix}, 
\end{equation*}
where $d_1 = a(1)-b(0)h_1$ and $d_2=a(N)-b(N)h_2$. It follows that $\pL$ is normal if and only if $L$ is normal. Next we introduce a necessary and sufficient condition for $L$ to be normal.
\begin{theorem}\label{equivcond}
Let $L =L_0 + \qi L_1 +\qj L_2+ \qk L_3\in M_N(\H)$ be a symmetric quaternion matrix with $L_0, L_1, L_2, L_3$ being real matrices. Then $L$ is normal if and only if $L_0$ is   commutative with $L_1, L_2, L_3$ respectively.
\end{theorem}
\begin{proof}
By direct computations, we have that $$LL^* - L^*L = 0 \Leftrightarrow LL^* - L^*L =  L^*L - LL^* \Leftrightarrow (L+L^*)(L-L^*)=(L-L^*)(L+L^*).$$
Since $L$ is symmetric, then $L_0, L_1, L_2, L_3$ are symmetric. It follows that $$L^* = L_0^{\top} -\qi L_1^{\top} -\qj L_2^\top  -\qk L_3^\top=L_0 -\qi L_1 -\qj L_2  -\qk L_3.$$
Therefore $L+L^*=2L_0$ and $L-L^*= 2\qi L_1+2\qj L_2  +2\qk L_3$. It follows that
 $L+L^*$ is commutative with $L-L^*$ if and only if $L_0$ is   commutative with $L_1, L_2, L_3$. The proof is complete.
\end{proof}


Let $\phi(k,\lambda,s)$ ($0\leq k \leq N+1$) be the unique solution of   (\ref{diffequ})--(\ref{BV1}) satisfying
$\phi(1,\lambda,s)=s$. By direct computations, we have that $\phi(k,\lambda,s)$ is a simple quaternion polynomial  with the form of $\sum_{j=0}^{k}c(j,k)s\lambda^{k}$ for 
$k\geq 1$, where $c(0,k), c(1,k),\dots, c(k,k)$ are  undetermined coefficients. 
Therefore, for $\phi(k,\lambda_0,s)$ to be the solution of (\ref{diffequ})--(\ref{BV2}), it is necessary that $\lambda_0$ is a zero of polynomial $\phi(N+1,\lambda,s)+h_2\phi(N,\lambda,s)$.

Let $$p_N(\lambda,s)=\phi(N+1,\lambda,s)+h_2\phi(N,\lambda,s)$$ and $$\bvphi(\lambda,s)=\left[\phi(1,\lambda,s),\phi(2,\lambda,s),\dots,\phi(N,\lambda,s)\right]^\top.$$ The relationship between $\bvphi(\lambda,s)$ and $L$  is described next.
%

\begin{theorem}\label{relation}
	Let $p_N(\lambda,s), \bvphi(\lambda,s), L$ be given above. Then  for $s_0\in\H\setminus\{0\}$ the following  assertions are equivalent.
	\begin{enumerate}
		\item $\bvphi(\lambda_0,s_0)$ is a solution of (\ref{diffequ})--(\ref{BV2}).
		\item  $L\bvphi(\lambda_0,s_0)=\bvphi(\lambda_0,s_0) \lambda_0$.
		\item   $p_N(\lambda_0,s_0)=0$.
	\end{enumerate}
\end{theorem}
\begin{proof}
If we  regard (\ref{diffequ})--(\ref{BV2}) as a system of equations, then it   involves $N+2$ equations and $N+3$ variables $x(0), x(1), \dots, x(N+1), \lambda$. Solving (\ref{BV1}) and (\ref{BV2})  with respect to $x(0)$ and $x(N+1)$  and plugging the results into (\ref{diffequ}) yields the equation $L\bx=\bx\lambda$. Thus statement 1 is equivalent to statement 2. Similarly if we solve (\ref{diffequ})--(\ref{BV1}) with respect to  $x(N)$ and $x(N+1)$ and plug the results into (\ref{BV2}), then  (\ref{BV2}) becomes $p_N(\lambda,s_0)=0$. It follows that statement 1 is equivalent to  statement 3. Thus the three assertions are equivalent.
\end{proof}

As a consequence of Theorem \ref{relation}, we obtain a corollary as follows.
\begin{corollary}\label{corroll1}
Let $p_N(\lambda,s), \bvphi(\lambda,s), L$ be given above. Then the following assertions hold.
\begin{enumerate}
 \item If $\bxi=(\xi_1,\xi_2,\dots,\xi_N)^{\top}$ is an eigenvector of $L$ corresponding to the eigenvalue $\lambda_0\in\H$, then $\xi_1\neq 0$, $\bvphi(\lambda_0,\xi_1)=\bxi$. Moreover,  $\bvphi(\lambda_0,\xi_1)$ is a solution of (\ref{diffequ})--(\ref{BV2})  and therefore $p_N(\lambda_0,\xi_1)=0$.
 \item If $p_N(\lambda_0,s_0)=0$ for $s_0\neq 0$, then  $\bvphi(\lambda_0,s_0)s_1= \bvphi(s_1^{-1}\lambda_0s_1,s_0s_1)$ for every $s_1\in\H\setminus\{0\}$.
\end{enumerate}
\end{corollary}
\begin{proof}
1. We mentioned previously that $x(1)$ in (\ref{diffequ})--(\ref{BV1}) cannot be zero, otherwise there is only a trivial solution for the equations. We can also see this fact from matrix $L$ directly. Since $L\bxi=\bxi\lambda_0$, then $d_1\xi_1+b(1)\xi_2=\xi_1\lambda_0$.  It follows that $\xi_1=0$ will result in $\xi_2=0$, as $b(k)\in\H\setminus\{0\}$ for $0\leq k\leq N$.  By similar arguments, we have $\xi_3=\xi_4=\cdots=\xi_N=0$ which leads to a contradiction with the  assumption that $\bxi$ is an eigenvector of $L$. Define $\xi_0=-h_1\xi_1$ and $\xi_{N+1}=-h_2\xi_N$, then
$x(k)=\xi_k$ ($0\leq k\leq N+1$) is a solution of  (\ref{diffequ})--(\ref{BV2}) satisfying 
$x(1)=\xi_1$. By the definition of $\bvphi$, we conclude that $\bvphi(\lambda_0,\xi_1)=\bxi$ for the uniqueness of the solution.

2. If $p_N(\lambda_0,s_0)=0$ for $s_0\neq 0$, then $L\bvphi(\lambda_0,s_0)=\bvphi(\lambda_0,s_0) \lambda_0$ by Theorem \ref{relation}. It follows that
\begin{equation*}
L\bvphi(\lambda_0,s_0)s_1=\bvphi(\lambda_0,s_0) s_1s_1^{-1}\lambda_0s_1.
\end{equation*}
Let $\bmeta=(\eta_1,\eta_2,\dots,\eta_N)^{\top}=\bvphi(\lambda_0,s_0)s_1$ and $\lambda_1=s_1^{-1}\lambda_0s_1$, then $L\bmeta=\bmeta \lambda_1$. By statement 1 and note that $\eta_1=s_0s_1$, we conclude that
\begin{equation*}
\bvphi(s_1^{-1}\lambda_0s_1,s_0s_1)=\bvphi(\lambda_1,\eta_1)=\bmeta =\bvphi(\lambda_0,s_0)s_1.
\end{equation*}
The proof is complete.
\end{proof}

\begin{remark}
From the statement 1 of Corollary \ref{corroll1}, we see that if $\lambda_0$ is an eigenvalue of $L$,    there always exists $s_0\neq 0$ such that $\bvphi(\lambda_0,s_0)$ is a solution of  (\ref{diffequ})--(\ref{BV2}), where $s_0$ can be the first element of any eigenvector of  $L$  corresponding to  $\lambda_0$. Additionally, we can further find $t_1, t_2, \dots$ such that 
$\bvphi(\lambda_0, t_1), \bvphi(\lambda_0, t_2), \dots$ are  solutions of (\ref{diffequ})--(\ref{BV2})  by the statement 2 of Corollary \ref{corroll1}. This can be seen by setting $t_k=s_0s_k$, where $s_k, \ k=1,2,\dots$ are   unequal nonzero real numbers. In particular, if  $\lambda_0$ is real, then $\bvphi(\lambda_0,s)$ is a solution of  (\ref{diffequ})--(\ref{BV2}) for every $s\in\H\setminus\{0\}$.
\end{remark}

The above result states that for any fixed  eigenvalue  $\lambda_0$, there exist infinitely many $t_1, t_2, \dots$ such that 
$\bvphi(\lambda_0, t_1), \bvphi(\lambda_0, t_2), \dots$ are  solutions of (\ref{diffequ})--(\ref{BV2}). 
 We will show that for any fixed $s\in\H\setminus\{0\}$, there exist $N$ distinct quaternion numbers $\lambda_1,\lambda_2,\dots,\lambda_N$ such that
$\bvphi(\lambda_1, s)$,  $\bvphi(\lambda_2, s)$,  $\dots$, $\bvphi(\lambda_N, s)$  are solutions of 
(\ref{diffequ})--(\ref{BV2}).
Furthermore, these solutions  constitute  an orthogonal basis of $\H^N$. 
\begin{theorem}\label{spetrumdiffequ}
Suppose that $\pL$ (or equivalently $L$) is normal.	Then for any fixed $s\in\H\setminus\{0\}$, there exist $N$ distinct quaternion numbers $\lambda_1,\lambda_2,\dots,\lambda_N$ such that
\begin{enumerate}
	\item $V=\left\{\bvphi(\lambda_1, s), \bvphi(\lambda_2, s),\dots, \bvphi(\lambda_N, s)\right\}$ is  an orthogonal basis of $\H^N$; 
	\item  every element of $V$ is a solution of (\ref{diffequ})--(\ref{BV2}), therefore
	$\lambda_k$ is an eigenvalue of $L$ for $k=1,2,\dots,N$.
\end{enumerate} 
\end{theorem}
\begin{proof}
	Since $L$ is normal, by Theorem \ref{spectral2}, there exists $U=[\bu_1,\bu_2,\dots,\bu_N]\in M_n(\H)$ such that $U^*LU=D$ where $D$ is a diagonal matrix with diagonal entries $\alpha_1, \alpha_2,\dots,\alpha_N\in \C$. Note that $U$ is unitary, we have $LU=UD$. It follows that $\bu_k=[u_{1k},u_{2k},\dots,u_{Nk}]\top$ is an eigenvector of $L$ corresponding to the eigenvalue $\alpha_k$ for $k=1,2,\dots,N$.    We may encounter   $\alpha_i=\alpha_j$ for some $i\neq j$, as it is possible that  $L$ has multiple standard eigenvalues. Thus we need to construct $\lambda_k$ from $\alpha_k$ so that $\lambda_1,\lambda_2,\dots,\lambda_N$ can be distinct.
	
	By Corollary \ref{corroll1}, we know that $u_{11},u_{12},\dots,u_{1N}$ are nonzero and 
	$ \bvphi(\alpha_k,u_{1k}) =\bu_k$ is a solution of (\ref{diffequ})--(\ref{BV2}).
	Let $t_1,t_2,\dots,t_N$ be the quaternion numbers such that 
	\begin{equation*}
	u_{11}t_1=u_{12}t_2=\cdots=u_{1N}t_N=s
	\end{equation*}
	and let $\lambda_k=t_k^{-1}\alpha_kt_k$ for $k=1,2,\dots,N$. It follows from Corollary \ref{corroll1} that
	\begin{equation*}
	\bvphi(\lambda_k, s) = \bvphi(t_k^{-1}\alpha_kt_k, u_{1k}t_k)= \bvphi(\alpha_k, u_{1k}) t_k = \bu_kt_k
	\end{equation*}
	is a solution of (\ref{diffequ})--(\ref{BV2}) for $k=1,2,\dots,N$. Note that the columns of $U$ form an orthonormal basis of $\H^N$, we have that
	$V$ is an orthogonal basis  and therefore $\lambda_1,\lambda_2,\dots,\lambda_N$  have to be unequal. Otherwise, suppose that  $\lambda_i=\lambda_j$ for some $i\neq j$, then $\bvphi(\lambda_i, s)=\bvphi(\lambda_j, s)$, which means  they cannot be  orthogonal.  	
\end{proof}

Having introduced several useful results about the solutions of   (\ref{diffequ})--(\ref{BV2}), we are in position to state the main result of this section: 
the sampling expansions associated with the quaternion difference equations.
\begin{theorem}\label{samplthm}
	Let  $\bvphi(\lambda,s)=\left[\phi(1,\lambda,s),\phi(2,\lambda,s),\dots,\phi(N,\lambda,s)\right]^\top$ be given above and suppose that $\pL$ is normal. For any fixed $s\in\H\setminus\{0\}$, let $f_s(\lambda)$ be a function defined by the discrete transform  
	\begin{equation}\label{Hmodule}
	f_s(\lambda):= \sum_{k=1}^{N} \overline{F(k)} \phi(k,\lambda,s),\quad F(k)\in\H,\quad k=1,2,\dots,N. 
	\end{equation}
	Then there exist  $\lambda_1,\lambda_2,\dots,\lambda_N\in \H$ which depend on $s$ such that
	\begin{equation}\label{samlingexpansion}
	f_s(\lambda) = \sum_{k=1}^{N}f_s(\lambda_k) \psi_k(\lambda,s)
	\end{equation}
	where $\psi_k(\lambda,s)=\frac{\left \langle\bvphi(\lambda_k,s), \bvphi(\lambda,s)\right\rangle}{\norm{\bvphi(\lambda_k,s)}^2}$.
\end{theorem}
\begin{proof}
	For any fixed $s\in\H\setminus\{0\}$, by Theorem \ref{spetrumdiffequ}, there exist 
	 $\lambda_1,\lambda_2,\dots,\lambda_N\in \H$ such that $V=\left\{\bvphi(\lambda_1, s), \bvphi(\lambda_2, s),\dots, \bvphi(\lambda_N, s)\right\}$ is  an orthogonal basis of $\H^N$. It follows that $${\bm F} =[F(1),F(2),\dots,F(N)]^\top$$ and  
	 $\bvphi(\lambda,s)$   have expansions in terms of $V$:
 \begin{align*}
{\bm F} &= \sum_{k=1}^{N}\frac{\bvphi(\lambda_k, s)}{\norm{\bvphi(\lambda_k,s)}^2}\langle \bvphi(\lambda_k, s),  {\bm F}\rangle=\sum_{k=1}^{N}\frac{\bvphi(\lambda_k, s)}{\norm{\bvphi(\lambda_k,s)}^2}\overline{f_s(\lambda_k)}, \\
\bvphi(\lambda,s)&=\sum_{k=1}^{N}\frac{\bvphi(\lambda_k, s)}{\norm{\bvphi(\lambda_k,s)}^2}\langle \bvphi(\lambda_k, s),  \bvphi(\lambda,s)\rangle 
= \sum_{k=1}^{N}\bvphi(\lambda_k, s)\psi_k(\lambda,s).
\end{align*}
It follows that
\begin{align*}
 f_s(\lambda)&= \langle {\bm F},  \bvphi(\lambda,s)\rangle \\
 & =  \sum_{k=1}^{N}f_s(\lambda_k)\langle \frac{\bvphi(\lambda_k, s)}{\norm{\bvphi(\lambda_k,s)}^2}, \sum_{j=1}^{N}\bvphi(\lambda_j, s)\psi_j(\lambda,s)\rangle \\
 &= \sum_{k=1}^{N}f_s(\lambda_k)\langle \frac{\bvphi(\lambda_k, s)}{\norm{\bvphi(\lambda_k,s)}^2}, \bvphi(\lambda_k, s)\psi_k(\lambda,s)\rangle\\
 &=\sum_{k=1}^{N}f_s(\lambda_k)\psi_k(\lambda,s),
\end{align*}	 
which completes the proof.	  
 \end{proof}
\begin{remark}
We denote by $H_s$ the totality of quaternion functions of form (\ref{Hmodule}). Then $H_s$ is a left $\H$-module and every   function in $H_s$ has at least one sampling expansion with form (\ref{samlingexpansion}).
\end{remark}

The   proof  for Theorem  \ref{samplthm} does not describe how to compute the sample  points   $\lambda_k$  and the interpolation functions $\psi_k(\lambda,s)$ for $k=1,2,\dots,N$. In the next part, some examples are presented to illustrate the sampling theorem.   We  discuss how to formulate the sampling expansions in practice and  show the process of computation in detail.
 
\begin{example}\label{exsampling1}
	\em  Consider the quaternion difference equation
	\begin{equation}\label{exam1difequ}
	-\qi x(k+1)+\qj x(k)-\qi x(k-1) = x(k) \lambda, \quad k=1,2,\dots,N
	\end{equation}
	with boundary condition 
	\begin{equation}\label{boudarycond1}
	x(0)=x(N+1)=0.
	\end{equation}
 It is easy to see that $\pL$ defined by (\ref{diffoperator}) is normal for (\ref{exam1difequ}). For simplicity, we only show  the result for $N=3$.  By direct computations, we have that $\phi(1,\lambda,s)=s$, $\phi(2,\lambda,s)=\qi s \lambda -\qk s$ and
	\begin{equation*}
	\phi(3,\lambda,s)= -s\lambda^2 -2s, \quad 	\phi(4,\lambda,s)= -\qi s\lambda^3 +\qk s\lambda^2-3\qi s\lambda+3\qk s.
	\end{equation*}
	Thus $p_3(\lambda,s) = \phi(4,\lambda,s) + 0 \phi(3,\lambda,s)=-\qi s\lambda^3 +\qk s\lambda^2-3\qi s\lambda+3\qk s$.

	 Let $s=s_1 = 1+\qk$. We consider the sampling expansion for 
	 \begin{equation*}
	 	f_{s_1}(\lambda):= \sum_{k=1}^{N} \overline{F(k)} \phi(k,\lambda,s_1).
	 \end{equation*}
	 We need to solve the equation
	\begin{equation*}
	p_3(\lambda,s_1)=(\qj-\qi)\lambda^3 +(\qk -1)\lambda^2+3(\qj-\qi)\lambda+3(\qk -1)=0.
	\end{equation*}
	By applying the algorithm in \cite{janovska2010note} for computing the zeros, we get the zero set  of $p_3(\lambda,s_1)$: $Z(s_1)=\{\qi\}\cup\theta(\qi\sqrt{3})$. It was shown in \cite{cheng2017novel} that   two eigenvectors of a normal operator are orthogonal if they correspond to two non-similar eigenvalues. So we may set $\lambda_1=\qi$.
	It is not easy to determine two elements  $\lambda_2,\lambda_3$ of $\theta(\qi\sqrt{3})\subset Z(s_1)$   such that $\langle\bvphi(\lambda_2, s_1),\bvphi(\lambda_3, s_1)\rangle=0$.
	We provide two methods to determine the sample points $\lambda_2,\lambda_3$.
	
	\textbf{Method 1:}  \emph{Computing eigenvectors(eigenvalues)  and applying quaternion Gram-
	Schmidt process.} We need  to  compute the eigenvectors of 
	\begin{equation*}
	L=\begin{pmatrix}
	 \qj&-\qi & 0\\ 
-\qi	& \qj&-\qi  \\
0	&    -\qi& \qj
	\end{pmatrix}.
	\end{equation*}
For $\lambda=\qi\sqrt{3}$, we can get two linearly $\H$-independent eigenvectors 
\begin{equation*}
 \bxi_1=\begin{pmatrix}
\qi\\ 
-\qi\sqrt{3}-\qj  \\
\qi	
\end{pmatrix},\quad  \bxi_2=\begin{pmatrix}
-\qi\sqrt{3}-\qj \\ 
2\qi	  \\
-\qi\sqrt{3}-\qj 	
\end{pmatrix}.
\end{equation*}	
It is  fortunate that the inner product $\langle\bxi_1,\bxi_2\rangle=-4\sqrt{3}$ is real. 
So we can construct an eigenvector $\bxi_3$  corresponding to $\qi\sqrt{3}$ from $\{\bxi_1,\bxi_2\}$ such that $\langle\bxi_1,\bxi_3\rangle=0$ by the quaternion Gram-Schmidt process \cite{farenick2003spectral}. 
Let  $\bxi_3=\bxi_2-\frac{\bxi_1}{\norm{\bxi_1}^2}\langle\bxi_1,\bxi_2\rangle=( \frac{-\qi}{\sqrt{3}}-\qj,\frac{-2\qj}{\sqrt{3}},\frac{-\qi}{\sqrt{3}}-\qj)^{\top}$, then  $\langle\bxi_1,\bxi_3\rangle=0$ and
\begin{align}
L \bxi_3& = L(\bxi_2-\frac{\bxi_1}{\norm{\bxi_1}^2}\langle\bxi_1,\bxi_2\rangle)\notag\\
& = \bxi_2\qi\sqrt{3} -\frac{\bxi_1\qi\sqrt{3}}{\norm{\bxi_1}^2}\langle\bxi_1,\bxi_2\rangle \label{Lex1}\\
& = (\bxi_2-\frac{\bxi_1}{\norm{\bxi_1}^2}\langle\bxi_1,\bxi_2\rangle)\qi\sqrt{3}=\bxi_3\qi\sqrt{3}. \notag
\end{align}
 In Eq. (\ref{Lex1}),  $\qi\sqrt{3} \langle\bxi_1,\bxi_2\rangle =  \langle\bxi_1,\bxi_2\rangle \qi\sqrt{3} $ because $\langle\bxi_1,\bxi_2\rangle$ is real. Otherwise, to find suitable $\bxi$ we need   repeating eigenvector computing and quaternion Gram-Schmidt process as the procedure of diagonalization for quaternion normal matrices (see Theorem 3.3 in \cite{farenick2003spectral}).
 
By Corollary \ref{corroll1}, we know that
\begin{equation*}
\bvphi(\qi\sqrt{3},\qi)=\bxi_1, \quad \bvphi(\qi\sqrt{3},\frac{-\qi}{\sqrt{3}}-\qj) =\bxi_3
\end{equation*}
 and they are solutions of the BVP (\ref{exam1difequ})--(\ref{boudarycond1}).
 Let 
 \begin{equation*}
 t_1 = \qi^{-1}s_1=\qj-\qi, \quad   t_2 =(\frac{-\qi}{\sqrt{3}}-\qj)^{-1}s_1 =\frac{1}{4} \left(\sqrt{3}+3\right)\qi+\frac{1}{4} \left(3-\sqrt{3}\right)\qj.
 \end{equation*}
 Then we can select $\lambda_2$ and $\lambda_3$   to be
 \begin{equation*}
 \lambda_2 =t_1^{-1}\qi\sqrt{3}t_1=-\qj\sqrt{3},\quad  \lambda_3 =t_2^{-1}\qi\sqrt{3}t_2=\frac{3\qi+\qj\sqrt{3}}{2}.
 \end{equation*}
It follows that
 \begin{equation*}
\bvphi(\lambda_2,s_1)=\begin{pmatrix}
 1+\qk\\ 
 1- \sqrt{3}-(1+\sqrt{3})\qk  \\
 1+\qk	
 \end{pmatrix},\quad  \bvphi(\lambda_3,s_1)=\begin{pmatrix}
1+\qk \\ 
\frac{ \sqrt{3}-1}{2}	+ \frac{1+ \sqrt{3}}{2}\qk \\
1+\qk 	
 \end{pmatrix}.
 \end{equation*}	
 Note that $\bvphi(\lambda_1,s_1)=(1+\qk,0,-1-\qk)^{\top}$ and by some direct computations, we obtain the interpolation functions:
 \begin{align*}
 \psi_1(\lambda,s_1)&=\frac{\lambda^2+3}{2}, \\
  \psi_2(\lambda,s_1)&=\frac{-\lambda^2+(\qi+\qj\sqrt{3})\lambda+\qk\sqrt{3}}{6},\\
  \psi_3(\lambda,s_1)&=\frac{-2\lambda^2-(\qi+\qj\sqrt{3})\lambda-\qk\sqrt{3}-3}{6}.
 \end{align*}
 \textbf{Method 2:}  \emph{The method of undetermined coefficients and solving simple quaternion polynomials.} 
 It is also  natural to  construct an orthogonal basis $V=\{\bvphi(\lambda_k,s_1):k=1,2,3\}$ from $\bvphi(\lambda,s_1)$ directly without solving
 $L\bxi=\bxi\lambda$. Let $\widetilde{\lambda}_1=\lambda_1=\qi, \widetilde{\lambda}_2=\qi\sqrt{3}$. To find a $\widetilde{\lambda}_3\in \theta(\qi\sqrt{3})$ such that $\bvphi(\widetilde{\lambda}_2,s_1)$ and $\bvphi(\widetilde{\lambda}_3,s_1)$ are orthogonal, we  simply need to solve
 \begin{equation*}
 \langle \bvphi(\widetilde{\lambda}_2,s_1) ,\bvphi(\widetilde{\lambda}_3,s_1)\rangle =-2\widetilde{\lambda}_3^2+(2-2\sqrt{3})\qi\widetilde{\lambda}_3-2\sqrt{3}=0.
 \end{equation*}
The solution set is  $\{\qi,-\qi\sqrt{3}\}$. Note that $\widetilde{\lambda}_3$ has to be in $\theta(\qi\sqrt{3})$. It follows that $\widetilde{\lambda}_3=-\qi\sqrt{3}$. By some direct computations, we obtain the corresponding interpolation functions:
 \begin{align*}
\widetilde{\psi}_1(\lambda,s_1)&=\frac{\lambda^2+3}{2}, \\
\widetilde{\psi}_2(\lambda,s_1)&=\frac{-(3+\sqrt{3})\lambda^2-2\sqrt{3}\qi\lambda-3(1+\sqrt{3})}{12},\\
\widetilde{\psi}_3(\lambda,s_1)&=\frac{(\sqrt{3}-3)\lambda^2+2\sqrt{3}\qi\lambda+3(\sqrt{3}-1)}{12}.
\end{align*}
 
 Having introduced the sampling expansions for $f_{s_1}(\lambda)$, we now discuss how to derive the sampling expansions for $f_{s_2}(\lambda)$  ($s_2\neq s_1$) by making use of the relationship between  $f_{s_1}(\lambda)$ and $f_{s_2}(\lambda)$. Without loss of generality, we let $s_2 = 2\qj $  and $t = s_1^{-1}s_2=\qi+\qj$, by Lemma \ref{zerospzs}, the zero set of $p_3(\lambda,s_2)$  is  $$Z(s_2)= \{z=t^{-1}z_1 t: z_1\in Z(s_1)\}=\{\qj\}\cup\theta(\qi\sqrt{3}).$$
 It follows from   Corollary \ref{corroll1} that  $W_1 = \{\bvphi(\beta_k,s_2): k=1,2,3\}$ is an orthogonal basis where 
 \begin{equation*}
 \beta_1 = t^{-1}\lambda_1 t = \qj, \quad \beta_2 = t^{-1}\lambda_2 t = -\qi\sqrt{3},\quad  \beta_3 = t^{-1}\lambda_3 t = \frac{\sqrt{3}\qi+3\qj}{2}.
 \end{equation*}
 Similarly, another orthogonal basis is $W_2 = \{\bvphi(\widetilde{\beta}_k,s_2): k=1,2,3\}$, where
  \begin{equation*}
 \widetilde{\beta}_1 = t^{-1}\widetilde{\lambda}_1 t = \qj, \quad \widetilde{\beta}_2 = t^{-1}\widetilde{\lambda}_2 t =  \qj\sqrt{3},\quad  \widetilde{\beta}_3 = t^{-1}\widetilde{\lambda}_3 t = -\qj\sqrt{3}.
 \end{equation*}
 The corresponding interpolation functions can be computed by the definition, we omit the details.
\end{example}
\begin{remark}
   Two sampling expansions are given for  $f_{s_1}(\lambda)$ by different  methods. In fact, regardless of which method we apply, we may obtain various sampling expansions by using different initial candidate of the  sample points in $\theta(\qi\sqrt{3})$. The non-uniqueness of     sampling expansions will be further discussed in Proposition \ref{nonunique}. Another observation of this example  is that the sampling expansions for $f_{s_2}(\lambda)$  ($s_2\neq s_1$) can be easily obtained by making use of the relationship between  $f_{s_1}(\lambda)$ and $f_{s_2}(\lambda)$.
\end{remark}

%

\begin{example}
	\em  Consider the BVP (\ref{diffequ})--(\ref{BV2}) with	 the coefficients as follows.
\begin{center}
			\extrarowheight=2pt
	\begin{tabular}{c|c|c|c|c|c|c|c|c}
		\hline 
		$a(1)$	& $a(2)$ &$a(3)$ &$b(0)$ &$b(1)$  &$b(2)$  & $b(3)$ &$h_1$  &$h_2$   \\ 
		\hline 
	$\qj$	&  $\qi$&  $-\qk$& $\qi+\qj$ &  $\sqrt{3}\qj$& $\qj-\qk$ & $1+\qj$ & $-\qk$ &$ -\qi$  \\ 
		\hline 
	\end{tabular} 
\end{center}
 Then we have $\phi(1,\lambda,s)=s$, $\phi(2,\lambda,s)=-\frac{\sqrt{3}}{3}\qj s\lambda-\frac{\sqrt{3}}{3}\qk s$, $\phi(3,\lambda,s) = \frac{\sqrt{3}}{6}(\qi-1)s\lambda^2+\frac{2\sqrt{3}}{3}(\qi-1)s$ and
 \begin{align*}
  \phi(4,\lambda,s) =& \frac{\sqrt{3}}{12}(\qk+\qj+\qi-1)s\lambda^3+ \frac{\sqrt{3}}{12}(1+\qi+\qj-\qk)s\lambda^2\\
  &+\frac{\sqrt{3}}{2}(\qk+\qj+\qi-1)s\lambda+\frac{\sqrt{3}}{6}(3+3\qi+\qj-\qk)s.
 \end{align*}
 It follows that
  \begin{align*}
p_3(\lambda,s) =& \phi(4,\lambda,s) +h_2 \phi(3,\lambda,s) \\
=&\frac{\sqrt{3}}{12}(\qk+\qj+\qi-1)s\lambda^3+ \frac{\sqrt{3}}{12}(3+3\qi+\qj-\qk)s\lambda^2\\
 &+\frac{\sqrt{3}}{2}(\qk+\qj+\qi-1)s\lambda+\frac{\sqrt{3}}{6}(7+7\qi+\qj-\qk)s.
 \end{align*}
 Let $s_1=-\qk$, the zero set of $p_3(\lambda,s_1)$ is $Z(s_1)=\{-\qi-\qj,-\qi+2\qj,-\qi-3\qj\}$. Let $\lambda_1=-\qi-\qj, \lambda_2 =-\qi+2\qj, \lambda_3=-\qi-3\qj$, then 
\begin{equation*}
\bvphi(\lambda_1,s_1)=\begin{pmatrix}
-\qk\\ 
-\frac{\sqrt{3}}{3}\qk  \\
\frac{\sqrt{3}}{3}(\qj+\qk)	
\end{pmatrix},\quad \bvphi(\lambda_2,s_1)=\begin{pmatrix}
-\qk\\ 
\frac{2\sqrt{3}}{3}\qk  \\
-\frac{\sqrt{3}}{6}(\qj+\qk)	
\end{pmatrix}, \quad  \bvphi(\lambda_3,s_1)=\begin{pmatrix}
-\qk \\ 
-\sqrt{3}\qk \\
-\sqrt{3}(\qj+\qk)	
\end{pmatrix},
\end{equation*}
 are orthogonal as $\lambda_1,\lambda_2,\lambda_3$ belong to different similarity orbits. By some direct computations,  we obtain the interpolation functions:
  \begin{align*}
 \psi_1(\lambda,s_1)&=\frac{\lambda^2+\qj\lambda+7-\qk}{6}, \\
 \psi_2(\lambda,s_1)&=\frac{-\lambda^2-4\qj\lambda+4\qk+2}{15},\\
 \psi_3(\lambda,s_1)&=\frac{-\lambda^2+\qj\lambda-\qk-3}{10}.
 \end{align*}
\end{example}

\begin{remark}
 If   the number of zeros  of $p_N(\lambda,s)$ is equal to $N$ (which means that $p_N$ has $N$ distinct eigenvalues belonging to $N$ different similarity orbits), it is easy to find $N$ distinct quaternion numbers $\lambda_1,\lambda_2,\dots,\lambda_N$ satisfying the conditions  of  the proposed sampling theorem. In particular, if $p_N(\lambda,s)$ has $N$ isolated zeros,  then  the sampling expansion (\ref{samlingexpansion})  for  $f_s(\lambda)$  is unique.   
\end{remark}

In Example \ref{exsampling1}, the number of zeros for $p_N(\lambda,s)$ is less than $N$. We use different methods to obtain two   sampling expansions for  $f_s(\lambda)$.  Roughly speaking, the non-uniqueness is  caused by the non-real similarity orbits in  the  zero set  of $p_N$. 
\begin{proposition}\label{nonunique}
If  $p_N(\lambda,s)$ has at least one spherical zero,   then  there are infinitely many sampling expansions for $f_s(\lambda)$.
\end{proposition}
\begin{proof}
	Case 1: If the number of zeros for $p_N(\lambda,s)$ is equal to $N$. Suppose that  $f_s(\lambda)$ has a sampling expansion at   $\lambda_1,\lambda_2,\dots,\lambda_N$.  Without loss of generality, assume that 
	 $\lambda_1$ is a spherical zero of $p_N(\lambda,s)$. Then for any non-real $\beta$, $f_s(\lambda)$ has a sampling expansion at   $\beta^{-1}\lambda_1\beta,\lambda_2,\dots,\lambda_N$.

	Case2:  If the number of zeros for $p_N(\lambda,s)$ is less than $N$. Then there exists at least one non-real similarity orbit  $\theta(\lambda_0)$ containing at least two sample points, where  $\lambda_0$ is a spherical zero of $p_N(\lambda,s)$. Suppose that  $f_s(\lambda)$ has a sampling expansion at   $\lambda_1,\lambda_2,\dots,\lambda_N$.  Without loss of generality, assume that $\lambda_1,\lambda_2,\dots,\lambda_{n_1} \in \theta(\lambda_0)$ with $n_1\geq 2$. Pick $\alpha_1\in \theta(\lambda_0)$ and $\alpha_1\neq \lambda_i$ for all $1\leq i\leq n_1$. We denote the solution set of   
	 \begin{equation*}\label{E1}
	 \langle \bvphi(\alpha_1,s) ,\bvphi(\alpha,s)\rangle =0  
	 \end{equation*}
	 by $S_1$.  Pick $\alpha_2\in S_1\cap\theta(\lambda_0)$ and denote by $S_2$ the solution set of 
	  \begin{equation*}\label{E2}
	 \langle \bvphi(\alpha_2,s) ,\bvphi(\alpha,s)\rangle =0. 
	 \end{equation*}
	 Then we can find $\alpha_3\in S_2\cap S_1\cap\theta(\lambda_0)$ such that $\bvphi(\alpha_1,s)$, $\bvphi(\alpha_2,s)$, $\bvphi(\alpha_3,s)$ are orthogonal. Repeating this process, we have $\alpha_1,\alpha_2,\dots,\alpha_{n_1} \in \theta(\lambda_0)$  such that 
	 $f_s(\lambda)$ has a sampling expansion at   $\alpha_1,\alpha_2,\dots,\alpha_{n_1},\lambda_{n_1+1},\lambda_{n_1+2},\dots,\lambda_N$. Since $\theta(\lambda_0)$ contains infinitely many elements, we can always find new sample points to construct   sampling expansions for $f_s(\lambda)$.  The proof is complete.
\end{proof}	 
	  
\section{Characteristic polynomials for  tridiagonal quaternion matrices}\label{S4}

For any complex or real matrix $B$, we can compute its eigenvalues by finding all zeros of polynomial $\det(B-\lambda I)$.
When considering quaternion matrix $A\in M_n(\H)$, however, we don't know whether there exists a quaternion polynomial $p_A$ of degree $n$ having the property that  $p_A(\lambda)=0$ if and only
if $\lambda$ is a   eigenvalue of $A$. It is difficult to find such a $p_A$ by the traditional  methods used in the complex case. The reasons are two-fold.
\begin{enumerate}
	\item The non-singularity of $A-\lambda I$ does not mean $A\bxi= \bxi\lambda $ for some $\bxi\neq 0$.
	\item Though $A-\lambda I$ is non-singular if and only if  $A\bxi= \lambda\bxi $ for some $\bxi\neq 0$,  $\det(A-\lambda I)$ is not a polynomial with respect to $\lambda$. Here, $\det$ is  a arbitrary definition of  quaternion determinant described in \cite{aslaksen1996quaternionic}.
\end{enumerate}

In general,  to obtain the right eigenvalues of $A\in M_n(\H)$, one needs to compute the  eigenvalues of the complex   adjoint   matrix $\chi_A$ \cite{zhang1997quaternions} of $A$, where $\chi_A$ is a $2n\times 2n$ complex matrices with the form
	\begin{equation*}
	\chi_A=\begin{pmatrix}
		A_1&A_2 \\ 
		- \overline{A_2}	& \overline{A_1}
	\end{pmatrix}.
\end{equation*} 
Here $A_1,A_2$   stem from the unique 
representation $A = A_1 + A_2\qj,\ A_1,A_2\in M_n(\mathbb{C})$.

Theorem \ref{relation} gives us a hint to compute the right  eigenvalues of tridiagonal symmetric quaternion matrices by finding the zeros of a simple quaternion polynomial. The following result is a direct consequence of Theorem \ref{relation}.
\begin{theorem}\label{characterpoly}
	 Let $A\in M_n(\H)$ be a  tridiagonal symmetric quaternion matrix, there exists a simple quaternion polynomial $p_n(A,z)$ such that $$\sigma(A) = \{\alpha^{-1}\lambda\alpha:\alpha\neq 0, \lambda\in Z_p\},$$
	 where $\sigma(A)$ and $Z_p$ are the right spectrum of $A$ and the zero set of  $p_n(A,z)$ , respectively.
\end{theorem}

\begin{example}
\em	Consider the eigenvalue problem for
	\begin{equation*}
	A=\left(\begin{array}{cccc}
	1&  1+\qi& 0 & 0 \\
	1+\qi&  \qi& 1+\qj & 0 \\
	0& 1+\qj &  \qj& 1+\qk \\
	0& 0 &  1+\qk& \qk
	\end{array}\right).
	\end{equation*}
On the one hand, by solving the quaternion difference equation associated with 	$A$, we can construct a quaternion polynomial  
	\begin{equation*}
p_4(A,z)=( {\qi+\qk })z^4+(3-\qi-\qj-\qk)z^3+(3-\qi-\qj-\qk)z^2+(\qj-3-3\qi-3\qk)z+1-4\qi+\qj.
\end{equation*}
On the other hand, by computing the eigenvalues of 
	\begin{equation*}
	\chi_A= \left(
	\begin{array}{cccccccc}
	1 & 1+\qi & 0 & 0 & 0 & 0 & 0 & 0 \\
	1+\qi & \qi & 1 & 0 & 0 & 0 & 1 & 0 \\
	0 & 1 & 0 & 1 & 0 & 1 & 1 & \qi \\
	0 & 0 & 1 & 0 & 0 & 0 & \qi & \qi \\
	0 & 0 & 0 & 0 & 1 & 1-\qi & 0 & 0 \\
	0 & 0 & -1 & 0 & 1-\qi & -\qi & 1 & 0 \\
	0 & -1 & -1 & \qi & 0 & 1 & 0 & 1 \\
	0 & 0 & \qi & \qi & 0 & 0 & 1 & 0 \\
	\end{array}
	\right),
	\end{equation*}
we can find the standard eigenvalues of $A$.

 Table \ref{table-ex} displays the  zeros of $p_4$ and the standard eigenvalues of $A$. We see that $z_k$ is similar with $\lambda_k$ for $k=1,2,3,4$, since they have the same real part and the same norm \cite{zhang1997quaternions}. It is known that $\gamma$ is an eigenvalue  if and only if $\gamma$ is similar with a standard eigenvalue. By the transitive property of quaternion similarity, we have that $\gamma$ is an eigenvalue of $A$  if and only if $\gamma$ is similar with a zero of $p_4(A,z)$. This validates the statement of Theorem \ref{characterpoly}.
	\begin{table}[ht]
	\centering
	\caption{The zeros of $p_4$ and the standard eigenvalues of $A$ computed via $\chi_A$.}{
		\begin{tabular}{|  c   | @{} c @{} |@{}  c@{} |   c   |}
			\hline
			\multicolumn{2}{|c|}{Zeros of $p_4$}&  \multicolumn{2}{c|}{Standard eigenvalues of $A$}\\
			\hline
			$z_1$  &  $-1.12826 -0.378569 \qi+0.22245 \qj-0.321633 \qk $&  ~~$-1.12826+0.544285 \qi$~~& $\lambda_1$ \\ \hline
			$z_2$  &  ~$-0.208978-0.433043 \qi+0.412505 \qj+0.129384 \qk $ ~ & ~ $-0.208978+0.611905 \qi$ ~ & $\lambda_2$ \\ \hline
			$z_3$  &  $  1.03613+1.08041 \qi +0.0906333 \qj+0.326603 \qk$ &  $1.03613   +1.13233 \qi$ & $\lambda_3$ \\  \hline
			$z_4$  &  $ 1.3011+1.8469 \qi +0.0828995 \qj  +0.843984 \qk $ &  $1.3011  +2.0323 \qi$ & $\lambda_4$ \\ \hline
	\end{tabular}}\label{table-ex}
\end{table}
\end{example}

 \section{Conclusions}\label{S5}
 
 In this paper,
we investigate  the relationship between sampling theory and quaternion difference equations. The sampling expansions associated with quaternion difference equations are derived. Through examples, we show the computational  techniques for the proposed formula  in detail. Additionally, we find the characteristic polynomials for  tridiagonal symmetric quaternion matrices.

\section*{Acknowledgment}
This work was supported by the Guangdong Basic and Applied Basic Research Foundation (No. 2019A1515111185), Young Innovative Talents Project of Guangdong (2019KQNCX156), the Science and Technology Development Fund, Macau SAR (No. FDCT/085/2018/A2), Natural Science Foundation of Zhejiang Province (No. LY20A010016),  National Natural Science Foundation of China (No. 11671176).

{\small

}
\end{document}